\documentclass{birkjour}

\usepackage[cp850]{inputenc}
\usepackage{amssymb}
\usepackage[mathscr]{eucal}
\usepackage{amscd}
\usepackage{amsmath}
\usepackage[all]{xy}
\usepackage{amscd}
\usepackage{mathtools}

\newcommand{\R}{\mathbb R}
\newcommand{\e}{\varepsilon}
\newcommand{\N}{\mathbb N}

 \def\Ext{\operatorname{Ext}}

\def\dens{\operatorname{dens}}
\def\cop{\operatorname{co}^{(p)}}
\def\supp{\operatorname{supp}}

\theoremstyle{plain}
\newtheorem{theorem}{Theorem}[section]
\newtheorem{problem}{Problem}
\newtheorem{prop}[theorem]{Proposition}

\newtheorem{corollary}[theorem]{Corollary}

\theoremstyle{remark}

\newtheorem{defn}[theorem]{Definition}

\newcommand{\adef}{\begin{defn}}
\newcommand{\zdef}{\end{defn}}

\title{Trimming the Johnson bonsai}

\author[Cabello]{F\'elix Cabello S\'anchez}
\address{Universidad de Extremadura\\
Instituto de Matem\'aticas IMUEx\\
06006 Badajoz, Spain.}
\email{fcabello@unex.es}

\author[Castillo]{Jes\'us M.\,F. Castillo}
\address{Universidad de Extremadura\\
Instituto de Matem\'aticas IMUEx\\
06006 Badajoz, Spain.}
\email{castillo@unex.es}

\author[Moreno]{Yolanda Moreno}
\address{Universidad de Extremadura\\
Instituto de Matem\'aticas IMUEx\\
07001 C\'aceres, Spain.}
\email{ymoreno@unex.es}

\thanks{This research has been supported in part by projects PID2019-103961GB-C21 and PID2023-146505NB-C21 funded by MINCIN}

\begin{document}
\noindent {\footnotesize To appear in Banach Journal of Mathematical Analysis}\vspace{30pt}

\dedicatory{To Bill Johnson, whose work, no matter if we talk of his papers, his live conversation or his relentless activity at MathOverflow, is the finest exemplification of the verses of Joan Walsh Anglund:
A bird does not sing because he has an answer / He sings because he has a song.}

\begin{abstract} We show that  if $p>1$ every subspace of $\ell_p(\Gamma)$ is an $\ell_p$-sum of separable subspaces of $\ell_p$, and we provide examples of  subspaces of $\ell_p(\Gamma)$  for $0<p\leq 1$ that are not even isomorphic to any $\ell_p$-sum of separable spaces, notably the kernel of any quotient map $\ell_p(\Gamma)\to L_1(2^{\Gamma})$ with $\Gamma$ uncountable. We involve the separable complementation property (SCP) and the separable extension property (SEP), showing that if $X$ is a Banach space  of density character $\aleph_1$ with the SCP then the kernel of any quotient map $\ell_p(\Gamma)\to X$  is a complemented subspace of a space with the SCP and, consequently, has the SEP.\end{abstract}

\maketitle


\section{Introduction}

This paper stems from different lines of research that somehow crossed after writing \cite[p. 14 and Theorem 7.2.3]{hmbst}:\medskip

\textbf{First line of research.} Our first interest is a theorem of  K\"othe that states that every complemented subspace of $\ell_1(\Gamma)$ is isomorphic to $\ell_1(I)$ for some $I\subset \Gamma$. \medskip

This result is well-known, even if well-known is a way of saying because, while the proof for $\Gamma$ countable is perfectly standard and can be seen in \cite{lindtzaf}, K\"othe's proof for $\Gamma$ uncountable \cite{kothe} is not easily accessible, not to mention that it is written in German. Other proofs for the uncountable case have been obtained by Orty\'nski \cite{orty}, Rosenthal \cite[Corollary on p.~29]{ros}, Rodr\'iguez Salinas \cite{salinas} and Finol and Wojtowicz \cite{finowoj}. The result above, in combination with the fact that each Banach space is a quotient of some $\ell_1(\Gamma)$,
 yields  that the spaces $\ell_1(I)$ are the only projective objects in the category of Banach spaces.\medskip

\noindent Orty\'nsky moves deeper in this last direction and shows in \cite[Theorem 1]{orty}\medskip

$(\bigstar)$\quad Given $0<p\leq1$, every complemented subspace of $\ell_p(\Gamma)$ is isomorphic $\ell_p(I)$ for some $I\subset \Gamma$, \medskip

\noindent which yields  that $\ell_p(I)$ are the only projective spaces in the category of $p$-Banach spaces. This result together with the Aoki--Rolewicz theorem (every quasinormed space can be given an equivalent $p$-norm for suitable $p\in (0,1]$) plus
the fact that every $p$-Banach space is a quotient of some $\ell_q(\Gamma)$ for each $0<q\leq p$ shows that every projective quasi-Banach space is finite-dimensional.

Orty\'nski's remarkable extension of K\"othe's result 
 occurs even if what is true for $p=1$ (every closed subspace of $\ell_1$ contains a subspace isomorphic to $\ell_1$ and complemented in $\ell_1$) is false for $p<1$ (see \cite{arica}). Indeed, he shows \cite[Proposition 2]{orty} that for uncountable $\Gamma$ every closed subspace of $\ell_p(\Gamma)$ with density character $\aleph<|\Gamma|$ contains a subspace isomorphic to $\ell_p(\aleph)$ and complemented in $\ell_p(\Gamma)$.

\medskip

\textbf{Second line of research.} The proof of Moreno and Plichko that $c_0(\Gamma)$ is ``automorphic'' for every $\Gamma$ (see \cite{moreplic} or  \cite[Theorem 7.2.3]{hmbst}) is based on a result of Johnson--Zippin that states that every subspace of $c_0(\Gamma)$ admits a decomposition into a $c_0$-sum of separable subspaces; see \cite[Lemma 2]{jz} or Proposition~\ref{prop} below for a more precise statement. (A Banach space $X$ is automorphic if every isomorphism between subspaces $A,B$ of $X$ such that $X/A$ and $X/B$ are infinite dimensional and have the same density character can be extended to an automorphism of $X$. The proof that $c_0$ automorphic is the seminal result of Lindenstrauss and Rosenthal \cite{lindrose} and is deeply rooted in Banach space theory, but once one knows that, the proof that $c_0(\Gamma)$ is automorphic for every $\Gamma$ is mostly combinatorial.)\medskip

At least half of the present authors were firmly convinced for some time that a similar decomposition would be true for subspaces of $\ell_p(\Gamma)$ for $0<p<\infty$, which would have led to unifying proofs for the results of K\"othe, Orty\'nski and  Moreno--Plichko, among others.

The harsh reality showed us that our intentions were doomed to failure from the outset. Indeed, while in Section~2 we prove that:
\medskip

$\bullet$ If $p>1$ every closed subspace of $\ell_p(\Gamma)$ is an $\ell_p$-sum of subspaces of $\ell_p$\medskip

\noindent (due to Johnson and Zippin, see their Remark~3),
in Section 3 we show that the ``result'' is strongly false in general:\medskip

$\bullet$ For $0<p\leq 1$ there are closed subspaces of $\ell_p(\Gamma)$ not even isomorphic to any $\ell_p$-sum of separable spaces.\medskip

Our examples are kernels of surjections $\ell_p(\Gamma)\to L_p(\mu)$, where $\mu$ is a ``nonseparable'' probability.\bigskip

\textbf{Third line of research.}
Since we cannot change our spots,
once established that $\ell_p(\Gamma)$ contains weird subspaces for $0<p\leq 1$,
we focus on the problem of relating the properties of the subspaces of $\ell_p(\Gamma)$ to those of the corresponding quotients. A saucerful of highly non-trivial results on that topic is available in \cite[Sections 5.3 and 10.4]{hmbst}.

An important property when nonseparable spaces are to be considered is the separable complementation property (SCP), see \cite[Section 4]{pliyo}: \adef A Banach space has the separable complementation property if every separable subspace is contained in a separable complemented subspace.\zdef

Probably the seminal result is that all reflexive spaces have the SCP, a result wrongly attributed in \cite{casa} to Lindenstrauss' memoir instead of to \cite{lindreflexive}, and that has been later extended to WCG or WLD spaces \cite[Chapter VI]{DGZ}. A comprehensive approach to the SCP, including many new results, can be followed through the papers \cite{pliyo, PY, fjp}. Perfectly aware that no one seems to know whether the SCP is inherited by complemented subspaces (cf. \cite[Problem 6.4]{fjp}), Figiel, Johnson and Pe\l czy\'nski introduce the following, more affordable version (see their Remark 3.3):

\adef A Banach space $X$ has the separable extension property (SEP) if given a separable subspace $E\subset X$ there is an operator $T: X\to X$ with separable range that fixes $E$ (meaning that $Tx=x$ for all $x\in E$).  \zdef

However, the SEP has only a minor role in \cite{fjp} and its connection with the SCP remains unexplored, being only obvious that complemented subspaces of spaces with the SCP have the SEP. In this way, the main result we will obtain in Section \ref{jump}, namely, that for $0<p\leq 1$
the kernel of any quotient map $Q:\ell_p(\Gamma) \to X$ is a complemented subspace of a space with the SCP when $X$ is a $p$-Banach space with the SCP and density character $\aleph_1$ can be reformulated as:\medskip

$\bullet$ For $0<p\leq 1$ the kernel of any quotient map $Q:\ell_p(\aleph_1) \to X$ has the SEP if $X$ is a $p$-Banach space with the SCP.\medskip

\noindent \emph{Notational conventions.} We consider only real scalars. If $S$ is a subset of a linear space $X$, then $\operatorname{span}(S)$ denotes the linear subspace generated by $S$ in $X$. If $X$ carries a topology, then $\operatorname{\overline{span}}(S)$ is the closure of $\operatorname{span}(S)$. Given a (nonempty) set $I$ and $0<p<\infty$ we write $\ell_p(I)$ for the space of those functions $x:I\to\R$ such that
$\|x\|_p =\left( \sum_i |x(i)|^p\right)^{1/p}<\infty$. For $1\leq p<\infty$ this defines a norm on $\ell_p(I)$, but only a quasinorm (a $p$-norm in fact) for $0<p<1$. Similarly, $c_0(I)$ denotes the space of functions $x:I\to\R$ such that for every $\e>0$ the set $\{i\in I: |x(i)|\geq \e\}$ is finite, while $\ell_\infty(I)$ is the space of bounded functions on $I$.
As usual, we equip $c_0(I)$ and $\ell_\infty(I)$ with the norm $\|x\|_\infty=\sup_{i\in I}|x(i)|$.
If $I=\N$ we just write $\ell_p$ and $c_0$. For $J\subset I$ we often consider $\ell_p(J)$ as the subspace of elements of $\ell_p(I)$ that vanish off $J$. This will not cause any confusion.\medskip

Assume that $\mathscr J$ is a family of pairwise disjoint subsets of $I$ and that for each $J\in\mathscr J$ one has a subspace $X_J\subset \ell_p(J)\subset\ell_p(I)$. Then, if $x 1_J$ denotes the pointwise product of $x$ by the characteristic function of the set $J$,
$$\ell_p(\mathscr J, X_J)= \left\{x\in\ell_p(I): \supp x\subset \bigcup\mathscr J,  x|_J\in X_J\;\forall \,J\in\mathscr J\right\}.
$$
This is coherent with the traditional notation for the $\ell_p$-sum of families: if $(Y_k)$ is a family of spaces  indexed by $k\in K$, then $\ell_p(K,Y_k)$ is the space of families $(y_k)_{k\in K}$, with $y_k\in Y_k$, such that $\sum_{k\in K}\|y_k\|^p<\infty$. 
Indeed, due to the peculiarities of the $\ell_p$-quasinorm (if $p\neq 2$, nobody cares about $p=2$) a subspace $X$ of $\ell_p(\Gamma)$ is isometric to $\ell_p(K,Y_k)$ if and only if there is a decomposition $\mathscr J=\{J_k: k\in K\}$ and subspaces $X_k\subset \ell_p(J_k)$
such that $X=\ell_p(\mathscr J, X_k)$, where ``$=$'' means ``is''. The \emph{density character} of a quasinormed space $X$, denoted $\dens(X)$, is the minimal cardinality of a dense subset of $X$. 
 The cardinal of a set $I$ is denoted by $|I|$ and $\aleph_1$ denotes the first uncountable cardinal. Since cardinals are also sets, we will write $I,J,\Gamma$ for sets and $\aleph$ and similar symbols for cardinals mostly for psychological reasons.

\section{The case $p>1$}

Let us first present a decomposition result for subspaces of $\ell_p(\Gamma)$ when $1<p<\infty$, resulting in a new proof that every complemented subspace of $\ell_p(\Gamma)$ is isomorphic to $\ell_p(I)$ for some $I\subset\Gamma$; these results are inspired by \cite{orty} and \cite{moreplic}:

\begin{prop}\label{prop} Let  $1<p<\infty$ and let $\Gamma$ be a set. Every closed subspace of $\ell_p(\Gamma)$ has the form $\ell_p(\mathscr J, X_J)$, where $\mathscr J$ is a decomposition of $\Gamma$ into countable subsets and $X_J\subset \ell_p(J)$ for every $J\in\mathscr J$.
\end{prop}

The proof is hidden in the claim inside the proof of \cite[Lemma 2]{jz}:\medskip

\noindent\textbf{Claim.}\hspace{3pt} Let $X$ be a closed subspace of $\ell_p(\Gamma)$ for $1<p<\infty$ or $c_0(\Gamma)$. For every countable subset $I \subset \Gamma$ there is a countable $J\subset \Gamma$ containing $I$ such that $x 1_J\in X$ for every $x\in X$.\medskip

For the sake of completeness let us provide a sketch of the argument for subspaces of $\ell_p(\Gamma)$: consider the separable space $X_I= \{x1_I : x\in X\}\subset \ell_p(I)$ and obtain a separable subspace $X_1\subset X$ such that the unit ball of $(X_{1})_I =\{x1_I: x\in X_1\}$ is dense in the unit ball of $X_I$. Let $I_1\supset I$ be a countable set containing the supports of the elements of $(X_{1})_I$. Obtain in the same form a separable subspace $X_2\subset X$ such that the unit ball of $(X_{2})_I =\{x 1_I: x\in X_2\}$ is dense in the unit ball of $(X_{I_1}$ and  a countable set $I_2\supset I_1$ containing the supports of the elements of $(X_{2})_I$ and continue inductively. Set $J=\bigcup I_n$.
Given a normalized $x\in X$ pick a seminormalized sequence $(x_n)$, with $x_n\in X_n$ such that $\lim \|(x -x_n)1_{I_n}\|=0$ and therefore $\lim \|x 1_J -x_n 1_{I_n}\|=0$. Thus, taking limits in the weak topology one gets
$$\lim x 1_J - x_n = \lim x1_J - x1_{I_n} + \lim x1_{I_n} - x_n = \lim x1_{I_n} - x_n =0$$
since $x1_{I_n} - x_n$ is a bounded disjointly supported sequence, hence weakly $1$-summable in $\ell_p(\Gamma)$.\medskip

This argument of Johnson and Zippin, designed to show that every subspace of $c_0(\Gamma)$ has the form $c_0(\mathscr J, H_J)$, also works, as the authors say, ``in any space which has an extended shrinking unconditional basis" (see also Proposition \ref{prop1} below); in particular, in $\ell_p(\Gamma)$ when $p>1$ as we saw above. With this in hand we can follow the proof of \cite{moreplic} and obtain Proposition \ref{prop}:
\begin{proof}[Proof of Proposition \ref{prop}]  Apply Zorn's lemma to the set of families of pairwise disjoint, countable subsets of $\Gamma$ satisfying the previous Claim and obtain a maximal such family $\mathscr M$. By maximality, $\Gamma \setminus \bigcup \mathscr M$ must be either finite or empty; so there is no loss of generality in assuming it is empty. If we call $X_M = \{ x1_M : x\in X\}$ for $M\in\mathscr M$, then $X=\ell_p(\mathscr M, X_M)$: the only point to prove is that every $x\in X$ belongs to some
$X_M$, which is contained in the Claim.\end{proof}

After the proposition the following facts are immediate:

\begin{corollary} Let $1<p<\infty$ and let
$I$ be a set. An infinite-dimensional, closed subspace of $\ell_p(I)$ of density  character $\aleph$ contains a subspace
isomorphic to $\ell_p(\aleph)$ and complemented in $\ell_p(I)$.
\end{corollary}
\begin{proof}
If $\aleph$ is countable the result is a well-known feature of $\ell_p$.  Otherwise, write the subspace in the form $\ell_p(\mathscr J, X_J)$, pick a norm one element in each $X_J$, and note that $\dens(\ell_p(\mathscr J, X_J))=|\mathscr J|=\aleph$. \end{proof}

\begin{corollary} Let $1<p<\infty$. Every complemented subspace of $\ell_p(I)$ is isomorphic to $\ell_p(\aleph)$ for some $\aleph\leq |I|$.\end{corollary}

\begin{proof} Let $X$ be an infinite-dimensional, complemented subspace of $\ell_p(I)$. Then $X$ contains a complemented copy of $\ell_p(\aleph)$, where $\aleph=\dens(X)$. At the same time, $X$ is complemented in $\ell_p(J)$, where $J=\{i\in I: x(i)\neq 0 \text{ for some } x\in X\}$. Clearly, $\aleph= |J|$. Applying Pe\l czy\'nski's decomposition method in the form ``if $A$ is isomorphic to a complemented subspace of $B$ and $B$ is isomorphic to a complemented subspace of $A$ and $A\simeq\ell_p(\N, A)$, then $A\simeq B$'' we obtain that $X$ is isomorphic to $\ell_p(\aleph)$.
\end{proof}

\section{The case $0<p\leq 1$}

When $0<p\leq 1$, the proof above works without significant changes for ``weak*-closed'' subspaces. Here, we consider the weak* topology induced by $c_0(\Gamma)$ in $\ell_1(\Gamma)$ and so in $\ell_p(\Gamma)$ for $0<p<1$ and we call a subspace $X\subset \ell_p(\Gamma)$ ``weak*-closed'' if the weak*-closure of the unit ball of $X$ ``stays'' in $X$.

\begin{prop}\label{prop1}
Let $0<p\leq 1$.
Every weak*-closed subspace of $\ell_p(\Gamma)$ has the form $\ell_p(\mathscr J, X_J)$, where $\mathscr J$ is a decomposition of $\Gamma$ into countable subsets and $X_J\subset \ell_p(J)$ for every $J\in\mathscr J$.
\end{prop}

Proposition \ref{prop} is false for $0<p\leq 1$ and $\Gamma$ uncountable. Indeed, a decomposition $X=\ell_p(\mathscr J, X_J)$ as suggested implies that $\ell_p(\Gamma)/X$ is isometric to  $\ell_p(\mathscr J, \ell_p(J)/X_J)$ and contains a copy of $\ell_p(\mathscr J)$, which is complemented if $p=1$.
On the other hand, any $p$-Banach space with density $\aleph$ (or less) is a quotient of $\ell_p(\aleph)$. So, if $\aleph$ is uncountable and $0<p\leq 1$, the kernel of any quotient map
$Q:\ell_p(\aleph)\to Y$ fails to be isometric to an $\ell_p$-sum of separable subspaces if $Y$ is either $\ell_q(\aleph)$ for $p<q<\infty$ or $c_0(\aleph)$. If $\aleph$ is the continuum we may take $Y=\ell_\infty$ to get the same conclusion.\medskip

The conclusion in Proposition \ref{prop} is not only false for $0<p\leq 1$: it is false even in the much harder to handle isomorphic context. Let us recall that given two quasi Banach spaces $X, Y$, the equation $\Ext(X, Y)=0$ means that every exact sequence
$ 0\to Y\to Z\to X\to 0$
splits; namely, that every operator $Y\to E$ can be extended to an operator $Z\to E$ if $Z$ contains $Y$ and $Z/Y$ is isomorphic to $X$; equivalently, every operator $A\to X$ lifts to $Z$ --- see \cite[Section~2.14 and Chapter~4]{hmbst}. In the following result ${2}^\aleph$ denotes the product of $\aleph$-many copies of the group with two elements, with its Haar measure (actually probability).

\begin{prop}\label{prop:counter} Let $\aleph$ be an uncountable cardinal and let $p\in(0,1]$. The kernel $H$ in any exact sequence
$$
\xymatrix{0\ar[r]  & H \ar[r]& \ell_p(\aleph)  \ar[r] & L_p({2}^\aleph)\ar[r] & 0}$$
is not isomorphic to any $\ell_p$-sum of separable spaces.\end{prop}

\begin{proof} We make first the Banach case $p=1$.
Assume that $H$ is isomorphic to $\ell_1(I,X_i)$ with each $X_i$ separable and let $w:H\to \ell_1(I,X_i)$ be an isomorphism.
For each $k\in I$ the composition
$$
\xymatrixcolsep{3.5pc}
\xymatrix{
X_k\ar[r]^-{\text{inclusion}} & \ell_1(I,X_i) \ar[r]^-{w^{-1}} & H \ar[r]^-{\text{inclusion}} & \ell_1(\aleph)
}
$$
defines an embedding that we may call $\imath_k$. If $J_k$ denotes the union of the supports of the elements in the range of $\imath_k$, which is a countable subset of $\aleph$, then we can regard $	\imath_k$ as an embedding of $X_k$ into $\ell_1(J_k)=\ell_1$. Clearly, we have $\|\imath_k:X_k\to \ell_1\|\leq \|w^{-1}\|$, while
$\|\imath_k^{-1}\|\leq \|w\|$. All these embeddings can therefore be amalgamated to get a new, joint embedding $\jmath: \ell_1(I,X_i)\to \ell_1(I,\ell_1 )$, which yields the exact sequence
$$
\xymatrix{0\ar[r]  & \ell_1(I, X_i) \ar[r]^{\jmath}& \ell_1(I, \ell_1)  \ar[r] & \ell_1(I , \ell_1/\imath_i(X_i))\ar[r] & 0}
$$
The operator $\jmath \;w$ can be extended to an operator $u$ as in the diagram
$$\xymatrix{
0\ar[r]  & H \ar[r]\ar[d]_{\jmath\; w}& \ell_1(\aleph)\ar@{..>}[dl]^u  \ar[r] & L_1({2}^\aleph) \ar[r] & 0\\
&\ell_1(I, \ell_1)
}$$
because of the isomorphism $\ell_1(I, \ell_1) \simeq \ell_1(I\times\N)$,  Lindenstrauss' lifting, which yields
$\Ext(L_1({2}^\aleph, \ell_1(I\times \N))=0$ (see \cite[3.7.3]{hmbst}) and the delicacies of homology (take $p=1$ in \cite[Theorem 4.2.2]{hmbst}). Thus, there is a commutative diagram
\begin{equation}\label{uve}
\xymatrix{
0\ar[r]  & H \ar[r]\ar[d]^w& \ell_1(\aleph)  \ar[r] \ar[d]^u & L_1({2}^\aleph)\ar[r] \ar[d]^{v\quad\text{ (induced)}}   & 0\\
0\ar[r]  & \ell_1(I, X_i) \ar[r]& \ell_1(I, \ell_1)  \ar[r] & \ell_1(I , \ell_1/\imath_i(X_i))\ar[r] & 0
}\end{equation}

The existence of this diagram is, however, contradictory.
We all know that $\Ext( L_1(2^\aleph), c_0)\neq 0$, which means that some operator $H\to c_0$ does not admit an extension to $\ell_1(\aleph)$, see \cite[Proposition 4.5.13]{hmbst}.
But $c_0$ is separably injective with constant 2 and so every operator $t: X_k\to c_0$ has an extension to $\ell_1$ whose norm is at most $2\|t\|\|\imath_k\| \|\imath_k^{-1}\|\leq 2\|t\| \|w\| \|w^{-1}\|$, which implies that every operator
$\ell_1(I, X_i)\to c_0$ has an extension to $\ell_1(I, \ell_1)$.\medskip

For $0<p<1$ the proof contains even more drama: If we assume that an isomorphism $w:H\to \ell_p(I,X_i)$ exists, then since
$\Ext\left(L_p(\mu), \ell_p(\Gamma)\right) =0$ in the category of $p$-Banach spaces for all measures $\mu$ and sets $\Gamma$ by \cite[1.4.14 and 5.1.22]{hmbst} the same argument  yields the analogue to diagram (\ref{uve}) namely
\begin{equation*} \xymatrix{
0\ar[r]  & H \ar[r]\ar[d]^w& \ell_p(\aleph)  \ar[r] \ar[d]^u & L_p({2}^\aleph)\ar[r] \ar[d]^{v\quad\text{ (induced)}}   & 0\\
0\ar[r]  & \ell_p(I, X_i) \ar[r]& \ell_p(I, \ell_p)  \ar[r] & \ell_p(I , \ell_p/\imath_i(X_i))\ar[r] & 0
}\end{equation*}
Popov proved in \cite{popov} that no nonzero operator from $L_p(2^\aleph)$ to a separable quasi-Banach space exists. Therefore, every operator from $L_p(2^\aleph)$ to any ``sum'' of separable spaces is zero. Hence $v=0$ and $u$ must take values in $\ell_p(I, X_i)$, which implies that $w^{-1}u$ is a projection of $\ell_p(\aleph)$ onto $H$. Thus, the upper row splits, something impossible because $L_p({2}^\aleph)$ is not a subspace of $\ell_p(\aleph)$: there is no nonzero operator $L_p({2}^\aleph)\to \ell_p(\aleph)$, didn't we mention that before?
\end{proof}

\section{A comparison between the arguments of K\"othe, Jameson and Orty\'nski and those of Johnson--Zippin}

The basic idea behind the arguments of K\"othe \cite{kothe} and Ort\'ynski \cite{orty}, made explicit by Jameson in his book \cite{GJO}, is: Let $0<p\leq 1$ and let $(z_i)_{i\in I}$ be a family in $\ell_p(\Gamma)$ such that 
$ \|\sum_i \lambda_i z_i\|\geq \varepsilon \left( \sum_i |\lambda_i|^p\right)^{1/p} $ for some $\varepsilon>0$ and all $\lambda_i$. Let $(y_i)$ be another family such that $\|z_i - y_i\|\leq \rho$, and let $Z$ and $Y$ be the subspaces spanned by $(z_i)$ and $(y_i)$, respectively.
The operator $\tau: Z\to \ell_p(\Gamma)$ given by $\tau(z_i)=y_i$ has quasinorm $\|\tau\|\leq (1+(\rho/\varepsilon)^{p})^{1/p}$ since
\begin{eqnarray*}\left\|\sum \lambda_i y_i\right\|^p &\leq& \left\|\sum \lambda_i( y_i-z_i)\right\|^p  + \left\|\sum \lambda_i z_i\right\|^p\\
&\leq& \sum |\lambda_i|^p\rho^p + \left\|\sum \lambda_i z_i\right\|^p\\
&\leq& \left(\frac{\rho^p}{\varepsilon^{p}} +1\right) \left\|\sum \lambda_i z_i\right\|^p \end{eqnarray*}
and satisfies the estimate $\|{\bf I}|_Z -\tau\|<\rho/\varepsilon$. Thus, if $\rho/\varepsilon<1$, $\tau$ is an isomorphism between $Z$ and $Y$. To this setting, Orty\'nski adds a few new items:
\medskip

\noindent{\bf(A)}\quad If $Z$ is complemented in $\ell_p(\Gamma)$ by a contractive projection $P$ then $\tau$ can be extended to an automorphism as follows: write $\ell_p(\Gamma)=Z\oplus\ker P$ and define $T(z+w)=\tau(z)+ w$, where $z\in Z, w\in\ker P$. One has
$$
\|(T-{\bf I})(z+w)\|=\|\tau(z)-z\|\leq \frac{\rho}{\varepsilon}\|z\|\leq \frac{\rho}{\varepsilon}\|z+w\|,
$$
so $\|T-{\bf I}\|\leq {\rho}/{\varepsilon}<1$ and $T$ is an automorphism of $\ell_p(\Gamma)$.

\medskip

\noindent{\bf(B)}\quad Consequently, if $Z$ is complemented by a contractive projection $P$ then also $Y$ is complemented by $TPT^{-1}$; note that $\ker Q= T[\ker P]$.

\medskip

\noindent{\bf(C)}\quad If $X\subset \ell_p(\Gamma)$ has $\dens(X) <|\Gamma|$ then for every $I\subset \Gamma$ with $|I|<\dens(X)$ there is some $x \in X$ with large  ``tail" outside $I$. 
More precisely: there exists $\varepsilon>0$ such that for every $I\subset\Gamma$ with $|I|<\dens(X)$ one can find a normalized $x\in X$ such that $\| x 1_{\Gamma \setminus I}\|>\varepsilon$.

This is true because its negation (``For every $\varepsilon>0$ there exists $I\subset \Gamma$ with $|I|<\dens(X)$ such that
$\| x 1_{\Gamma \setminus I}\|\leq \varepsilon$ for every normalized $x\in X$'')
 cannot hold; otherwise, if $I_n$ is a set that works for $\varepsilon = 1/n$, picking $I = \bigcup_n I_n$ we get that for all normalized $x\in X $ and for every $n\in \N$ one has $\| x 1_{\Gamma \setminus I}\|\leq \| x 1_{\Gamma \setminus I_n}\|\leq 1/n$; since $n$ is arbitrary, this implies $x1_{\Gamma\setminus I}=0$; that is, $x1_I=x$; hence $X\subset \ell_p(I)$, which is in contradiction with $|I|=\dens(\ell_p(I))<\dens(X)$.\medskip

Thus, Orty\'nski gets his result using (C) to obtain a large family of disjoint vectors to which apply (A) and (B) to produce the copy of $\ell_p(\aleph)$ inside $X$. However, steps (A) and (B) fail when $p>1$ and thus what Johnson and Zippin need to do is to produce the copy of $\ell_p(\aleph)$ inside $X$ directly by modifying (C) as in the Claim: $x1_J\in X$ (but not equal to $x$). In fair return, no such thing can happen when $0<p\leq 1$.

\section{Enter the separable complementation property}\label{jump}

Ok, good: when $0<p\leq 1$ there are subspaces $H$ of $\ell_p(\aleph)$ not isomorphic to any $\ell_p$-sum of separable spaces; and this happens even if the quotient $\ell_p(\aleph)/H$ is a good-natured space of type $L_p(\mu)$. In particular, there are subspaces of $\ell_1(\aleph)$ not isomorphic to weak*-closed subspaces. How bad can those kernels be?

The reader is addressed to \cite{fjp,pliyo,PY} to find a rather complete picture of what is known and unknown about the SCP already defined in the Introduction. In particular, that all the spaces $L_p(\mu)$ enjoy the SCP for $p\in[1,\infty)$: half of the authors of this paper think that this easily follows from the existence of conditional expectations and that there are less clear proofs; and the other half think that some of those less clear proofs, that WCG spaces have the SCP \cite[Chapter IV, Lemma 2.4]{DGZ} for instance, are clearer. It is clear anyway that $L_1(\mu)$ spaces for finite $\mu$ are WCG since the inclusion $L_2(\mu)\to L_1(\mu)$ has dense range. Be as it may, $L_1(\mu)$ spaces for arbitrary $\mu$ have the SCP. That is why we were interested in the following result.

\begin{theorem}\label{th:SCP} Let $0<p\leq 1$. If $X$ is a $p$-Banach space with the SCP and density character $\aleph_1$, and $Q: \ell_p( \aleph) \to X$ is a quotient map then $\ker Q$ is a complemented subspace of a space with the SCP.\end{theorem}

\begin{proof} The proof has been organized in sections, labeled with roman numbers, just for the sake of clarity.
The reader is warned that the argumentation is indirect: we will first show that a certain space (a kind of quasilinear envelope of $X$) has the SCP and then we will deduce our statement.\medskip

\noindent{\bf (I)} \, If $\omega_1$ denotes the first uncountable ordinal, one has  $X=\bigcup_{\alpha<\omega_1} E_\alpha$, where $(E_\alpha)$ is an increasing $\omega_1$-family of separable subspaces; we do not require any  ``continuity'' condition, just that $E_\alpha\subset E_\beta$ for $\alpha<\beta<\omega_1$. We can use the SCP to get a similar family of complemented subspaces that we may call $X_\alpha$. Indeed, assuming that $X_\alpha$ has been defined for $\alpha<\beta<\omega_1$, take any complemented subspace of $X$ containing both $E_\beta$ and the closure of $\bigcup_{\alpha<\beta} X_\alpha$ and call it $X_\beta$. \medskip

 \noindent{\bf (II)} \, Once we have $X=\bigcup_{\alpha<\omega_1} X_\alpha$ it is easy to construct by transfinite induction a Hamel basis $\mathscr H$ for $X$ such that for each $\alpha<\omega_1$, $\mathscr H\cap X_\alpha$ is a Hamel basis of $X_\alpha$. 
Fix that Hamel basis for the rest of the proof and reserve. Each $x\in X$ can be written as a (finite) sum
$$
x=\sum_{h\in \mathscr H}\lambda_h(x) h,
$$
with the extra property that if $x\in X_\alpha$ and $\lambda_h(x)\neq 0$ then $h\in X_\alpha$. \medskip

\noindent{\bf (III)} \, The following remark is trivial: if $Z$ is a quasi-Banach space generated by a set $S$ (in the sense that no proper closed subspace of $Z$ contains $S$) and $Z_0$ is a closed, separable subspace of $Z$, then there exists a countable $S_0\subset S$ such that $Z_0\subset \overline{\operatorname{span}}(S_0)$.\medskip

 \noindent{\bf (IV)} \, We now switch to \emph{quasilinear} mode. Let $A, B$ be $p$-normed spaces, $p\in(0,1]$. A homogeneous map $\Phi: A\to B$ is $p$-linear if it obeys an estimate
$$
\left\|\Phi\left( \sum_{k\leq n} a_k\right) - \sum_{k\leq n} \Phi(a_k)\ \right\| \leq Q
\left( \sum_{k\leq n}\| a_k\|^p\right)^{1/p}
$$
for some constant $Q\geq 0$, every $n\in\N$ and all $a_k\in A$. Roughly speaking, if $A,B$ are $p$-Banach spaces, then $p$-linear maps $A\to B$ correspond to exact sequences of $p$-Banach spaces $0\to B\to C\to A\to 0$, see \cite[Section 3.6]{hmbst} for details. The least possible constant $Q$ for which the preceding inequality holds is $Q(\Phi)$, the ``$p$-linearity constant'' of $\Phi$.   \medskip

 \noindent{\bf (V)} \, An important feature of $p$-linear maps is the following: if $A_0$ is a dense subspace of $A$ and $B$ is complete, then every $p$-linear map $\Phi: A_0\to B$ can be extended to $A$ keeping $p$-linearity. The $p$-linearity constant may vary, however; see \cite[Section 3.9]{hmbst}. \medskip

 \noindent{\bf (VI)} \, We invoke \cite[Lemma 3.9.1]{hmbst} with the purpose of, starting with a $p$-linear map $\Phi: A\to B$ and a finite dimensional subspace $F\subset A$  and  finitely many points $x_1, \dots, x_n$ in (the unit sphere of) $F$,
  obtaining a $p$-linear map $\Phi': A\to B$ such that:
\begin{itemize}
\item $\|\Phi(a)-\Phi'(a)\|\leq 2Q(\Phi)\|a\|$ for every $a\in A$; 
\item $\Phi'(x_k)=\Phi(x_k)$ for $k\leq n$;
\item $\Phi'(F)$ is finite-dimensional (it is contained in a finite-dimensional subspace of $B$).
\end{itemize}
For the sake of peace of mind of all of us (readers and authors) let us explain how $\Phi'$ is obtained from $\Phi$, following the proof of \cite[Lemma 3.9.1]{hmbst}. First, $\Phi'=\Phi$ off $F$. To define $\Phi'$ on $F$ we 
first enlarge the set $\{x_1, \dots, x_n\}$ to a $\delta$-net in the unit sphere of $F$, which we still call the same,
 where $\delta<(1-2^{-p})^{1/p}$.
After that, one writes each $x$ in the sphere of $F$ as a linear combination $x=\sum_{k\leq n} c_k x_k$ with a good control of $\sum_{k\leq n} |c_k|^p$ according to some rules to decide which linear combination must be picked and sets
$$
\Phi'(x)=\sum_{k\leq n}c_k\Phi(x_k).
$$
Finally, one extends $\Phi'$ to $F$ by homogeneity. Since one of the rules we have just mentioned is that the obvious combination is chosen if $x$ happens to be some $x_k$, it turns out that $
\Phi'$ agrees with $\Phi$ at each $x_k$. \medskip

 \noindent{\bf (VII)} \, With this tool at hand we can obtain the following ``separable version'': 
 Let  $\Phi: A\to B$ be a $p$-linear map, $S$ be  countable
 subset of $A$ and $A_\infty=\operatorname{span}(S)$. Then there exists a countable $S'\subset A_\infty$ and a $p$-linear map
$\Psi: A_\infty\to B$ such that:
\begin{itemize}
\item $\Psi$ agrees with $\Phi$ on $S\cup S'$; 
\item $\|\Psi(a)-\Phi(a)\|\leq M\|a\|$ for some $M$ and all $a\in A_\infty$;
\item  $\Psi[A_\infty]\subset \operatorname{span}\Phi(S\cup S')$. In particular $\Psi[A_\infty]$ is (contained in some) separable (subspace of $B$).
\end{itemize}
This is achieved applying (VI) inductively, and keeping in mind that ``what has been already done remains done''.

Precisely, let $(x_n)_{n\geq 1}$ be an enumeration of $S$ and
consider the chain of finite-dimensional subspaces $A_n=\operatorname{span}\{x_1,\dots,x_n\}$. Note that $A_\infty =\bigcup_n A_n$.

Now choose an auxiliary sequence $(x_k')$ so that for every $n\geq 2$ the set
$$
\{x_1, x_1',\dots,x_{k(2)}', x_2,  x_{k(2)+1}',\dots, x_{k(n-1)}' x_{n-1},
 x_{k(n-1)+1}',\dots, x_{k(n)}', x_n\}
$$
is a $\delta$-net of the sphere of $A_n$,
and put $S'=\{x_k': k\geq 1\}$.

For each $n\geq 2$ let
 $\Phi_n': A\to B$ be the map one obtains from (VI) taking $F=A_n$ and using the $\delta$-net three lines above.
 For $x\in A_\infty$, put $n(x)=\min\{n: x\in A_n\}$ and then define
$$
\Psi(x)=\begin{cases} \Phi(x) & \text{if $x\in S\cup S'$;}\\ \Phi_{n(x)}'(x) &\text{otherwise}.\end{cases}
$$

The map $\Psi:A_\infty\to B$ is $p$-linear since one actually has $\|\Phi(x)-\Psi(x)\|\leq 2Q(\Phi)\|x\|$ for every $x\in A_\infty$ and, by the very definition, $\Psi(x)\in \operatorname{span}\Phi(S\cup S')$ for every $x\in A_\infty$, so  $\Psi[A_\infty]= \operatorname{span}\Phi(S\cup S')$.\medskip

According to (V) this $\Psi$  admits a $p$-linear extension $\overline{A_\infty}\to
 \operatorname{ \overline{span}} \,\Phi(S\cup S')$ which we still denote by $\Psi$.\medskip

 \noindent{\bf (VIII)} \,We also summon from
\cite[Section 3.10]{hmbst} the construction of the ``universal $p$-linear map'' $\mho$: given a $p$-Banach space $X$ and a Hamel basis $\mathscr H$ of $X$ there is a $p$-Banach space $\cop(X)$ together with a $p$-linear map $\mho: X \to \cop(X)$ that vanishes on $\mathscr H$ and with the universal property that for every $p$-linear map $\Phi: X\to Y$ vanishing on $\mathscr H$ there is a unique linear operator $\phi: \cop(X)\to Y$ such that $\Phi = \phi \mho$. The ``points'' $\mho(x)$, with $x$ normalized in $X$, generate a dense subspace of $\cop(X)$. \medskip
The core of the proof is the following piece:

\medskip

\noindent \textbf{Claim.} If $X$ is a $p$-Banach space with the SCP and density character $\aleph_1$, then $\cop(X)$ has the SCP as well.

\medskip

\noindent\emph{Proof of the Claim.}
\,
Write $X=\bigcup_{\alpha<\omega_1} X_\alpha$ as in (I) and then fix a Hamel basis $\mathscr H$ as in (II). Observe that different Hamel bases of $X$ just lead to isometric representations of $\cop X$, so we can choose the one that suits us best.\smallskip

Let $Y$ be a separable subspace of $\cop(X)$. According to (III) we may assume that $Y$ is generated by
a sequence $(y_k)$, with $y_k=\mho(z_k)$ and $z_k$ normalized in $X$. Take $\alpha<\omega_1$ large enough so that $X_\alpha$ contains all the $z_k$, and let $(x_n)_{n\geq 1}$ be a sequence dense in the unit sphere of $X_\alpha$ and containing $(z_k)_k$ as a subsequence.\smallskip

We define a countable subset $S\subset X_\alpha$ as follows:
\begin{itemize}
\item $x_n\in S$ for every $n$;
\item if $\lambda_h(x_n)\neq 0$ for some $n$, then $h\in S$.
\end{itemize}
That is, $S$ contains the vectors $x_1,x_2,\dots$, together with the minimal subset of $\mathscr H$ whose linear combinations contain those vectors. Note that $S\subset X_\alpha$ because $\mathscr H\cap X_\alpha$ is a Hamel basis of $X_\alpha$, that the closure of $A_\infty=\operatorname{span} (S)$ in $X$ is $X_\alpha$ and that if $h\in\mathscr H$ is such that $\lambda_h(a)\neq 0$ for some $a\in A_\infty$, then $h\in S$: actually $\mathscr H\cap A_\infty$ is  a (Hamel) basis of $ A_\infty$.\smallskip

Apply (VII) in its full literality, and then (V) to get a countable $S'\subset A_\infty \subset X_\alpha$ and a $p$-linear map  $\Psi: X_\alpha \to\cop(X)$ such that:
\begin{itemize}
\item $\Psi$ agrees with $\mho$ on $S\cup S'$; in particular, $\Psi(x_n)=\mho(x_n)$ for all $n$, and $\Psi(h)=0$ at each  $h\in \mathscr H\cap A_\infty$.
\item $\Psi[S\cup S']$ spans a separable subspace of $\cop(X)$, whose closure we will call $Y'$, and $Y'$ contains $\Psi[X_\alpha]$.
\end{itemize}

We are almost there. Let $P$ be a projection of $X$ onto $X_\alpha$. Then
$\Psi\circ P $ is $p$-linear from $X$ to $Y'$. Let $L$ be the only linear map from $X$ to $Y'$ that agrees with $\Psi\circ P$ on $\mathscr H$. Then $\Phi= \Psi\circ P-L$ is $p$-linear from $X$ to $Y'$ and vanishes on $\mathscr H$, so by the universal property of $\cop(X)$ there is an operator $\phi: \cop(X)\to Y'$ such that $\phi\circ\mho = \Phi$.\smallskip

 But $L$ vanishes on $A_\infty$ because $\mathscr H\cap A_\infty$ is a basis of $A_\infty$ contained in $S'$. Hence, if $h\in \mathscr H\cap A_\infty$ we have $L(h)=\Psi(P(h))=\Psi(h)=\mho(h)=0$.\smallskip

We finally check that $\phi$ is a projection: since $Y'= \operatorname{\overline{span}}\mho[S\cup S']$ it suffices to verify that it fixes the values $\Psi(x)=\mho(x)$ for $x\in S\cup S'$, which is clear since for such an $x$, one has
$\phi(\Psi x) = \phi(\mho x) = \Psi P(x) - L(x)  = \Psi(x)$.

\hfill{\emph{End of the proof of the Claim.}}\medskip

To conclude the proof, observe the diagram
\begin{equation*}
\xymatrix{
0\ar[r]  & \ker Q \ar[r]& \ell_p(\aleph)  \ar[r] & X \ar[r] & 0\\
0\ar[r]  & \cop(X) \ar[r]& \cop(X) \oplus_{\mho} X \ar[r] & X\ar[r]\ar@{=}[u] & 0}\end{equation*}
By the projective character (the ``lifting property'') of $\ell_p(\aleph)$ and the universal property of $\mho$ it turns out that each of those sequences is a pushout of the other (cf. \cite{hmbst}, Section~3.10 and Corollary 3.10.3). Applying then the  ``diagonal principle'' \cite[Proposition 2.7.3]{hmbst} one gets
$$ \ker Q \oplus \left( \cop(X) \oplus_{\mho} X\right) \simeq \cop(X) \oplus \ell_p(\aleph).$$

Since products of (finitely many) spaces with the SCP have it and $\ell_p(\aleph)$ has the SCP, we are done. \end{proof}

It is a shameless shortcoming of our proof that it apparently does not work for $\dens(X)> \aleph_1$.  Anyway, it yields:

\begin{corollary} The kernel of any quotient map $Q: \ell_p(\aleph)\to L_1(2^{\aleph_1})$ with $p\in(0,1]$  has the SEP.\end{corollary}

\section{Closing time}

As mentioned in the introduction, a wealth of results in (quasi) Banach space theory adopt the following blueprint: \emph{Let $Q:\ell_p(I)\to X$ be a quotient map, where $p\in(0,1]$. If $X$ has a certain property $\mathscr P$, then $\ker Q$ has $\mathscr P$ --- or a property closely related to $\mathscr P$.}

Theorem~\ref{th:SCP} is a modest contribution to this line of research. Results roving the opposite direction tend to be difficult: What happens to $\ker Q$ when $X$ lacks the SCP (or the SEP)? In particular,

\begin{problem}[Plichko--Yost, cf. closing remarks of \cite{PY}]
\emph{Does the kernel of a quotient map $Q:\ell_1(I)\to\ell_\infty$ have the SCP or the SEP? }
\end{problem}

Also:

\begin{problem}[Plichko--Yost, loc. cit.]
\emph{Does $\ell_1(I)$ contain a subspace without the SCP?}
\end{problem}

 Figiel, Johnson and Pe\l czy\'nski have shown in \cite{fjp} the remarkable result that if $Q:\ell_1\to c_0$ is a quotient map, then $(\ker Q)^{**}= \ker Q^{**}$ lacks the SCP in spite of being a subspace of $\ell_1^{**}$, who has the SCP because it is an $L_1(\mu)$-space in disguise. To the best of our knowledge that is the first (and sole) known example of a Banach space with the SCP that contains a subspace that lacks it.\medskip

On the other hand, the counter-example of Proposition~\ref{prop:counter} is very specific since it critically depends on the homological properties of the spaces $L_1(\mu)$ and so we cannot rule out the possibility that
the kernel of a quotient map $\ell_1(I)\to\ell_\infty$ is isomorphic to an $\ell_1$-sum of separable subspaces. Nor do we know whether the space $H$ in Proposition~\ref{prop:counter} admits a Schauder decomposition into separable subspaces --- recall that a Schauder decomposition of $X$ is a family $(X_\alpha)$ of subspaces of $X$, indexed by $\Gamma$, such that:

\begin{itemize}
\item[(a)] Every $x\in X$ can be written in a unique way as $x=\sum_\alpha x_\alpha$, with $x_\alpha\in X_\alpha$.
\item[(b)] There exists a contant $M$ such that if $x=\sum_\alpha x_\alpha$, with $x_\alpha\in X_\alpha$, and $A\subset \Gamma$ is finite, then $\|\sum_{\alpha\in A} x_\alpha\|\leq M\|x\|$.
\end{itemize}

If the series in (a) is unconditionally converging the decomposition is said to be unconditional.

\begin{problem}
\emph{Prove or disprove that every subspace of $\ell_1(I)$ that admits a  Schauder decomposition into separable subspaces is isomorphic to an $\ell_1$-sum of separable subspaces. What if the decomposition is unconditional?}
\end{problem}

Note that subspaces of the $\ell_1$-sum could be different from those of the Schauder decomposition: $\ell_1$ contains subspaces with (unconditional) bases that are not isomorphic to $\ell_1$.\medskip

The space $H$ of Proposition~\ref{prop:counter} cannot have a symmetric (uncountable) basis: Troyanski \cite[Corollary 2]{troy} showed that a space with symmetric basis of cardinal $\aleph$ that contains $\ell_1(\kappa)$ for uncountable $\kappa$ is isomorphic to $\ell_1(\aleph)$. In our case, $H$ cannot be even complemented in its bidual because otherwise the exact sequence $$ \xymatrix{0\ar[r]  & H \ar[r]& \ell_1(\aleph)  \ar[r] & L_1({2}^\aleph)\ar[r] & 0}$$ sequence splits by Lindenstrauss' lifting \cite[3.7.3]{hmbst}. Drewnowsky \cite[Theorem 1]{drew} revisits Troyanski's result above and shows that when the basis $(u_\alpha)$ of $H$ is merely unconditional then one can (only) get that a subfamily of $(u_\alpha)$ of cardinality $\kappa$ equivalent to the unit basis of $\ell_1(\kappa)$.\medskip

The only reason behind the irritating cardinality assumption in Theorem \ref{th:SCP} is to ``find'' a Hamel basis $\mathscr H$ of $X$ and a cofinal system of separable, complemented subspaces $(X_\alpha)$ such that $\mathscr H\cap X_\alpha$ is a Hamel basis for $X_\alpha$. This suggests the following half-combinatorial half-cooked question that does not even deserve a full stop: Let $X$ be a Banach space with the SCP. It is possible to find a family of uniformly complemented, separable subspaces $X_\alpha$ and a Hamel basis of $X$ such that:
\begin{itemize}
\item Every separable subspace of $X$ is contained in some $X_\alpha$;
\item $\mathscr H\cap X_\alpha$ is a Hamel basis for $X_\alpha$ for every $\alpha$?
\end{itemize}

Be as it may, the restriction is vexatious to the spirit. Therefore:

\begin{problem}
\emph{ Show that if $X$ has the SCP and $Q: \ell_p( \aleph) \to X$ is a quotient map then $\ker Q$  has SEP.}
\end{problem}

As mentioned by Figiel, Johnson and Pe\l czy\'nski in \cite{fjp}, it is not known whether the SCP is inherited by complemented subspaces. If this were true, but even if not, it could happen that the kernels considered in the previous problem have the SCP. We don't even know if that is true for $\aleph=\aleph_1$.

\begin{problem}\emph{
Show that if $X$ has the SCP and $Q: \ell_p( \aleph) \to X$ is a quotient map then $\ker Q$  has SCP.}
\end{problem}

Regarding the connections between SCP and SEP:\medskip

\begin{problem}
\emph{Is every Banach space with the SEP a complemented subspace of a space with the SCP?}
 \end{problem}

We conjecture that the answer is affirmative.
Or, as Orihuela hinted at the ``New perspectives in Banach spaces" conference at Castro Urdiales, 8--12 July 2024,

\begin{problem}
\emph{Does the SCP agree with the SEP?}
 \end{problem}

In a different direction, one may regard both the SEP and the SCP as ``separable approximation properties''. It would be interesting to know whether some traditional variant of the approximation property (meaning a condition defined by means of finite-rank operators) implies the SCP. This problem is discussed by Zolk in \cite{zolk}, explaining why the current status of \cite[Theorem 9.3]{casa} is up in the air.

\end{document}